\documentclass[12pt]{article}
\usepackage[cmtip,arrow]{xy}
\usepackage{pb-diagram, pb-xy}
\usepackage{amssymb,epsfig,amsfonts,amsmath}

\usepackage{color}


\newtheorem{nt}{Remark}
\newtheorem{Th}{Theorem}

\newfont{\ssdbl}{msbm8}
\newfont{\sdbl}{msbm9}
\newfont{\dbl}{msbm10 at 12pt}

\newcommand{\Aut}{\mathop {\rm Aut}}

\newcommand{\Z}{\dz}

\newcommand{\Ker}{{\rm Ker}\:}

\newcommand{\lto}{\longrightarrow}

\def\Z{{\mathbb Z}}

\begin{document}

\author{
D.V. Osipov}

\title{Discrete Heisenberg group and its automorphism group
\thanks{This paper was written with the partial financial support of RFBR (grants no. ~13-01-12420 ofi{\_}m2, 14-01-00178 a)}}

\date{}

\maketitle

In this note we give more easy and short proof of a statement  previously proved by P.~Kahn in~\cite{Ka} that the automorphism group of the discrete Heisenberg group  ${\rm Heis}(3, \Z) $
is isomorphic to  the group $ (\Z \oplus \Z)  \rtimes GL(2,\Z)$. The method which  we suggest to construct this isomorphism
gives far more transparent picture of the structure of the automorphism group of the group   ${\rm Heis}(3, \Z) $.

We consider the discrete Heisenberg group $G = {\rm Heis}(3, \Z) $ which is the group of matrices of the form:
$$
\left(
\begin{array}{rcl}
1 & a & c \\
0 & 1 & b \\
0 & 0 & 1
\end{array}
\right) \mbox{,}
$$
where $a, b, c$ are from $\Z$. We can consider also the group $G$ as a set of all integer triples endowed with the group law:
\begin{equation}  \label{comp}
(a_1,b_1,c_1)(a_2,b_2,c_2)= (a_1 + a_2, b_1 +b _2, c_1+ c_2 + a_1b_2)   \mbox{.}
\end{equation}
It is clear that $[(1,0,0), (0,1,0)]= (0,0,1)$
and the group $G$ is generated by the elements $(1,0,0)$
and  $(0,0,1)$.

Let $\Aut(G)$ be the group of automorphisms of the group $G$.
The starting point for this note was the group homomorphism $\Z \to \Aut(G)$ constructed in~\cite[\S~5]{OP} and given by the formula
\begin{equation}  \label{expl-auto}
R_d ((a,b,c))=  \left(a +db, b, c + \frac{b(b-1)d}{2} \right)  \mbox{,}
\end{equation}
where $R_d$ is the automorphism of the group $G$ induced by an element $d \in \Z$.  Formula~\eqref{expl-auto}
was obtained in~\cite{OP} after some calculation of automorphisms which are analogs of ''loop rotations'' in loop groups.
But the fact that formula~\eqref{expl-auto} defines a homomorphism from the group $\Z$ to the group  $\Aut(G)$ is also an easy  direct consequence of formulas~\eqref{comp} and~\eqref{expl-auto}.

\smallskip

For the group $G $ we have the following exact sequence of groups
\begin{equation}   \label{center}
1 \lto C \lto G \stackrel{\lambda}{\lto} H \oplus P \lto 1  \mbox{,}
\end{equation}
where $1$ is the trivial group, the group $C = \{(0,0,c) \mid c \in \Z \} \simeq \Z$ is the center of $G$, and $H = \{(a,0,0) \mid a \in \Z  \} \simeq \Z$, $P =\{(0,b,0) \mid b \in \Z\}  \simeq \Z$.

From exact sequence~\eqref{center} we obtain the following sequence
\begin{equation}  \label{aut}
1 \lto {\rm Inn} (G)  \stackrel{\theta}{\lto}  {\rm Aut} (G) \stackrel{\vartheta}{\lto} GL(2,\Z) \mbox{,}
\end{equation}
where the group of inner automorphisms ${\rm Inn}(G) \simeq G/C \simeq \Z \oplus \Z$, and the homomorphism $\vartheta$ is the homomorphism ${\rm Aut}(G)  \to {\rm Aut(\Z \oplus \Z)}$, which is induced by the homomorphism $\lambda$ from exact sequence~\eqref{center}.
It is clear that $\mathop{\rm Im}(\theta)  \subset \Ker(\vartheta)$.

We claim that sequence~\eqref{aut}   is exact in the term ${\rm Aut}(G)$. Indeed, it is enough to prove that $\Ker(\vartheta) \subset
\mathop{\rm Im} (\theta)$. Consider any  $\omega \in \Ker(\vartheta)$. We have $\omega ((1,0,0))= (1,0, c_1)$ and $\omega ((0,1,0))=(0,1,c_2)$
for some integer $c_1$ and $c_2$. By direct calculations, we obtain $(c_2, -c_1, 0) (1,0,0)= (1,0,c_1) (c_2, -c_1, 0)$  and
${(c_2, -c_1, 0) (0,1,0)= (0,1, c_2) (c_2, -c_1, 0)}$. Since elements $(1,0,0)$ and $(0,1,0)$
generate the group $G$,
we obtain that the inner automorphism defined by the element $(c_2, -c_1, 0)$ coincides with the automorphism $\omega$.

We claim also that the homomorphism~$\vartheta $ from sequence~\eqref{aut} is surjective.
Indeed, by formula~\eqref{expl-auto} we have a homomorphism from the group $\Z$ to the group $\mathop{\rm Aut}(G)$. It is easy to see that this homomorphism is a section of the homomorphism $\vartheta$ over the subgroup
$\left\{\left(
\begin{array}{rl}
  1 & d \\
 0 & 1
\end{array}
\right) \right\}_{d \in \Z}$  of the group $GL(2,\Z)$. By formula~\eqref{expl-auto},  the action of the matrices from this subgroup  on elements of the group $G$ is given as:
\begin{equation}  \label{upper-hom}
\left(
\begin{array}{rl}
  1 & d \\
 0 & 1
\end{array}
\right)
(a,b,c) = \left(a +db, b, c + \frac{b(b-1)d}{2} \right)  \mbox{.}
\end{equation}

Symmetrically to the formula~\eqref{upper-hom} we can  write the following formula:
\begin{equation}  \label{bottom-hom}
\left(
\begin{array}{rl}
  1 & 0  \\
 d & 1
\end{array}
\right)
(a,b,c) = \left(a, da + b, c + \frac{a(a-1)d}{2} \right)  \mbox{.}
\end{equation}
By an easy direct calculation we have that formula~\eqref{bottom-hom}
defines a correct automorphism  of the group $G$. This automorphism depends on $d \in \Z$ and
defines the homomorphism from the subgroup  $\left\{\left(
\begin{array}{rl}
  1 & 0 \\
 d & 1
\end{array}
\right) \right\}_{d \in \Z}$ of the group $GL(2,\Z)$ to the group $\mathop{\rm Aut}(G)$. This homomorphism defines a section of the homomorphism
$\vartheta$ over this subgroup.

Besides, it is easy to see that the homomorphism from the subgroup $\left\{\left(
\begin{array}{rl}
  \pm 1 & 0 \\
 0 & 1
\end{array}
\right) \right\}$ of the group $GL(2, \Z)$ to the group $\mathop{ \rm Aut}(G)$  given by the formula
\begin{equation}  \label{minus-auto}
\left(
\begin{array}{rl}
  - 1 & 0 \\
 0 & 1
\end{array}
\right) (a,b,c) = (-a,b, -c-b)
\end{equation}
defines a section of
the homomorphism
$\vartheta$ over this subgroup.

By the classical result (see its proof, for example, in~\cite[Appendix~A]{KT}), the group $SL(2, \Z)$ has a presentation:
$$
<\rho ,\tau  \; \mid \; \rho \tau \rho= \tau \rho \tau  \mbox{,}  \;  (\rho \tau \rho)^4 =1 >  \mbox{,}
$$
where the element $\rho$ corresponds to the matrix $A =\left(
\begin{array}{rl}
  1 & 1 \\
 0 & 1
\end{array}
\right) $ and the element $\tau$ corresponds to the matrix $B = \left(
\begin{array}{rl}
  1 & 0 \\
 -1 & 1
\end{array}
\right) $.
The group $GL(2,\Z)$ is generated by the elements of the group $SL(2,\Z)$ and the matrix $ D =\left(
\begin{array}{rl}
  - 1 & 0 \\
 0 & 1
\end{array}
\right) $.
Therefore the group $GL(2,\Z)$ has a presentation:
\begin{equation}   \label{present}
<\rho ,\tau, \kappa  \; \mid \; \rho \tau \rho= \tau \rho \tau  \mbox{,}  \;  (\rho \tau \rho)^4 =1  \mbox{,}  \;
\kappa \tau \kappa^{-1} = \tau^{-1}  \mbox{,}  \; \kappa \rho \kappa^{-1} = \rho^{-1} \mbox{,}  \; \kappa^2=1
>  \mbox{,}
\end{equation}
where the element $\kappa$ corresponds to the matrix $D$.
Thus we obtained that the homomorphism  $\vartheta$ is surjective.

We note that the group $\mathop{ \rm Aut}(G)$ contains a distinguished subgroup $\mathop{\rm Aut}^+(G) = \vartheta^{-1} (SL(2, \Z))$
of index $2$. Since for any $ \omega  \in \mathop{ \rm Aut}(G)$ we have
$$
\omega((0,0,1))= \omega ([(1,0,0), (0,1,0)])= [\omega(1,0,0), \omega(0,1,0)]=  (0,0, \det(\vartheta(\omega))  )  \mbox{,}
$$
we obtain that the group $\mathop{\rm Aut}^+(G) $  consists of elements $\omega$ of the  group $\mathop{ \rm Aut}(G)$ such that $\omega((0,0,1))=(0,0,1)$. In other words, the group $\mathop{\rm Aut}^+(G)$ consists of automoprhisms of the group $G$ which act identically on the center of the group $G$.

\begin{Th} \label{semidirect}
Partial sections of the homomorphism $\vartheta$ given by formulas~\eqref{upper-hom}, \eqref{bottom-hom} and~\eqref{minus-auto} are glued together and define a section of the homomorphism $\vartheta$ over the whole group $GL(2,\Z)$. Hence, the group $\mathop{\rm Aut}(G) \simeq (\Z \oplus \Z)  \rtimes GL(2,\Z)$,
where an action of $GL(2,\Z)$ on $\Z \oplus \Z$ (given by inner automorphisms in the group $\mathop{\rm Aut}(G)$) is the natural matrix action.
Besides, $\mathop{\rm Aut}^+(G) \simeq (\Z \oplus \Z)  \rtimes SL(2,\Z)$.
\end{Th}
\begin{proof}  \hspace{-0.3cm}.
From the above discussion we see that it is enough to check that automorphisms of the group $G$ which are given by the matrices
$A =\left(
\begin{array}{rl}
  1 & 1 \\
 0 & 1
\end{array}
\right)$, $B =\left(
\begin{array}{rl}
  1 & 0 \\
 -1 & 1
\end{array}
\right) $  and  $ D =\left(
\begin{array}{rl}
  - 1 & 0 \\
 0 & 1
\end{array}
\right) $  (with the help of formulas~\eqref{upper-hom}, \eqref{bottom-hom} and~\eqref{minus-auto}) satisfy relations from formula~\eqref{present}. Since the group $G$ is generated by elements $(1,0,0)$  and $(0,1,0)$, it is enough to check these relations  between  automorphisms after application to these elements. We have
\begin{eqnarray*}
ABA (1,0,0) = AB (1,0,0)= A(1,-1,0)= (0,-1,1)  \mbox{, and} \\
BAB (1,0,0)= BA (1,-1,0)= B(0,-1,1)= (0,-1,1)  \mbox{.}
\end{eqnarray*}
Besides, we have
\begin{eqnarray*}
ABA (0,1,0) = AB (1,1,0)= A(1,0,0)= (1,0,0)   \mbox{, and} \\
BAB (0,1,0)= BA(0,1,0)= B(1,1,0)= (1,0,0)  \mbox{.}
\end{eqnarray*}
Thus we see that the automorphism of the group $G$ given by the composition $ABA$ coincides with the automorphism of $G$ given by the
composition $BAB$. We will check now that the composition $(ABA)^4$ defines the identical automorphism of the group $G$. Indeed, we have checked that $ABA (1,0,0)= (0,-1,1)= (0,1,0)^{-1}$. Therefore, using again above calculations, we have
$$
(ABA)^2(1,0,0) = ABA ((0,1,0)^{-1})= (ABA (0,1,0))^{-1}= (1,0,0)^{-1}  \mbox{.}
$$
Hence we obtain
$$
(ABA)^4(1,0,0)= (ABA)^2 ((1,0,0)^{-1})= (1,0,0)  \mbox{.}
$$
Analogously we have
\begin{eqnarray*}
(ABA)^2(0,1,0)= ABA (1,0,0)= (0,1,0)^{-1}  \mbox{, and hence} \\
(ABA)^4 (0,1,0)= (ABA)^2((0,1,0)^{-1})= (0,1,0)  \mbox{.}
\end{eqnarray*}
Besides, we calculate
\begin{eqnarray*}
DAD^{-1}(0,1,0)= DA(0,1,-1)= D(1,1,-1)= (-1,1,0)= A^{-1}(0,1,0)  \mbox{, and} \\
DAD^{-1}(1,0,0)= DA(-1,0,0)= D(-1,0,0)= (1,0,0)= A^{-1}(1,0,0)  \mbox{, and} \\
DBD^{-1}(0,1,0)= DB (0,1,-1)= D(0,1,-1)= (0,1,0)=B^{-1}(0,1,0)  \mbox{, and}  \\
DBD^{-1} (1,0,0)= DB (-1, 0, 0)= D(-1, 1,-1 )= (1,1,0)= B^{-1}(1,0,0)  \mbox{.}
\end{eqnarray*}
We have also directly from formula~\eqref{minus-auto} that the automorphism of the group $G$ given by the composition $D^2$
is the identity automorphism.
The theorem is proved.

\end{proof}

\begin{nt}
P.~Kahn constructed in~\cite{Ka}  a section of the homomorphism $\vartheta$ over the  group $GL(2,\Z)$ by direct calculations inside the group $\mathop{\rm Aut}(G)$ (in contrast to our approach, which we started from the explicit homomorphism given by formula~\eqref{expl-auto}  and this homomorphism was obtained in~\cite[\S~5]{OP} from an analog of "loop rotations").  We note, that the section $GL(2,\Z) \to \mathop{\rm Aut}(G)$ constructed by P.~Kahn
does not coincide with the section from Theorem~\ref{semidirect}.
\end{nt}

    Let $\alpha_1 $ and $\alpha_2$ be two sections of the homomorphism $\vartheta$. We define $\varphi(g) = \alpha_2(g)
    \alpha_1(g)^{-1}  \in \Z \oplus \Z$ for any $g \in GL(2,\Z)$. Then $\varphi(g_1 g_2)=
    \varphi(g_1) + g_1 \cdot \varphi(g_2) $ for any $g_1, g_2 \in GL(2,\Z)$, i.e. the map $\varphi :
    GL(2,\Z)  \to \Z \oplus \Z$ is a $1$-cocycle. Conversely, for any section $\alpha$ of the homomorphism
    $\vartheta$ and for
     any $1$-cocycle $ \varphi :
    GL(2,\Z)  \to \Z \oplus \Z$ we have that the map $(\varphi \cdot \alpha) : GL(2,\Z)  \to \Aut(G)$ is a homomorphism which is a section of the homomorphism
    $\vartheta$. Using the Mayer-Vietoris sequence for the presentation $SL(2,\Z) = \Z_4 *_{\Z_2} \Z_6$ as amalgamated free product P.~Kahn computed in~\cite{Ka} that $H^1(SL(2, \Z), \Z \oplus \Z) =0$ for the natural matrix action of $SL(2,\Z)$ on $\Z \oplus \Z$. Hence, from the Lyndon-Hochschild-Serre spectral sequence applied to $SL(2,\Z) \hookrightarrow GL(2,\Z)$ one immediately obtains that $H^1(GL(2, \Z), \Z \oplus \Z)=0$. Therefore $\varphi$ is a $1$-coboundary, i.e. there is an element $a \in \Z \oplus \Z$
    such that $\varphi (g)= ga -a$ for any $g \in GL(2,\Z)$.
    Hence we obtain that the group of $1$-cocycles of  $GL(2,\Z)$ with values in  $\Z \oplus \Z$
    is isomorphic to the group $\Z \oplus \Z$. Thus, the set of all sections of the homomorphism $\vartheta$ is a principal homogeneous space for the group $\Z \oplus \Z$.

\vspace{0.3cm}

\noindent
Steklov Mathematical Institute of RAS \\
Gubkina str. 8, 119991, Moscow, Russia \\
{\it E-mail:}  ${d}_{-} osipov@mi.ras.ru$ \\

\end{document}